\documentclass[letterpaper, 10pt, conference]{ieeeconf}      

\IEEEoverridecommandlockouts                             

\usepackage[letterpaper,left=54pt,right=54pt,top=54pt,bottom=54pt]{geometry}

\usepackage{graphics} 
\usepackage{epsfig} 
\usepackage{times} 
\usepackage{amsmath} 
\usepackage{amssymb}  
\usepackage{tikz}
\usepackage{mathtools}
\usepackage{comment}

\def\P{\mathcal{P}}
\def\C{\mathcal{C}}
\def\B{\mathcal{B}}

\allowdisplaybreaks[4]

\newtheorem{lemma}{Lemma}[section]
\newtheorem{definition}{Definition}[section]

\title{\LARGE \bf
An ADMM Algorithm for MPC-based Energy Management in Hybrid Electric Vehicles with Nonlinear Losses
}

\author{Sebastian East and Mark Cannon
\thanks{This work was funded by the UK Engineering and Physical Sciences Research Council.}%
\thanks{S.~East and M.~Cannon are with the Department of Engineering Science, University of Oxford, OX1 3JP, {\tt\footnotesize sebastian.east@eng.ox.ac.uk},}%
\thanks{{\tt\footnotesize mark.cannon@eng.ox.ac.uk}}%
}

\begin{document}

\maketitle
\thispagestyle{empty}
\pagestyle{empty}

\begin{abstract}

In this paper we present a convex formulation of the Model Predictive Control (MPC) optimisation for energy management in hybrid electric vehicles, and an Alternating Direction Method of Multipliers (ADMM) algorithm for its solution. We develop a new proof of convexity for the problem that allows the nonlinear dynamics to be modelled as a linear system, then demonstrate the performance of ADMM in comparison with Dynamic Programming (DP) through simulation. The results demonstrate up to two orders of magnitude improvement in solution time for comparable accuracy against DP.

\end{abstract}

\section{INTRODUCTION}
Automotive powertrains have become increasingly electrified in recent decades, in an effort to reduce both hydrocarbon consumption and tailpipe emissions \cite{Malikopoulos2014}. Hybrid electric vehicles are a common configuration as they can reduce fuel consumption through regenerative braking, engine shut down, and reduced engine size, but the second power source introduces an additional degree of freedom that makes the powertrain underconstrained, so the fraction of power delivered by the engine and motor must be controlled. The total energy consumed during a journey can be significantly reduced by varying this fraction \cite{Sciarretta2007}, but a control system that is optimal (in the sense of minimising fuel consumption) whilst also being implementable is still an open problem \cite{MarinaMartinez2016}.

Model Predictive Control (MPC) is a control framework where for each control input update, the solution of an open-loop finite-horizon optimal control problem is solved online, and the first element of the optimal predicted control vector is implemented as the current control input. It has been applied to the hybrid vehicle energy management problem as the optimisation can explicitly account for constraints on system variables, and the recursive feedback provides a degree of robustness to modelling, measurement, and disturbance uncertainty \cite{Buerger,Xiang2017}.

An important issue with MPC is the online solution of the optimal control problem: it must capture the important dynamics of the controlled system, but the optimisation problem may be high-dimensional, constrained, and possibly nonconvex. Dynamic Programming (DP) has been used to provide the globally minimising argument \cite{Sun2015a}, however the approach is too slow to be implemented in real time and is computationally expensive. Conversely, linearisation has been used to reduce the optimisation to a quadratic program \cite{Borhan2012a}, which is a well-known problem and can be solved rapidly, however the linearisation process renders the applied control inputs sub-optimal. To alleviate these difficulties, the problem can be reformulated as a convex optimisation \cite{Egardt2014,Hadj-Said2016,Nuesch2014}. This allows the dynamics of the system to be accurately modelled using convex functions, and enables a guarantee of convergence to the global minimum. 

The first contribution of this paper is an alternative proof of convexity for the energy management problem. Previous work has achieved convexity by relaxing the system dynamic constraints to inequalities, and demonstrating that the optimal solution is obtained when the dynamic constraints are satisfied with equality \cite{Egardt2014}. We instead show that the problem is convex when expressed in terms of battery power output. This modification allows the dynamics to be represented without approximation as a linear system, which simplifies the implementation of convex optimisation algorithms.

The second contribution is an Alternating Direction Method of Multipliers (ADMM) \cite{Wang2014} algorithm for the solution of the convex problem. Our work in \cite{Buerger} demonstrated a projected Newton method that only considers the terminal state of charge constraint, and other work in the area has made use of generic convex optimisation solvers; the method presented here enforces state of charge constraints at all times, and is specifically designed to exploit separability of the cost and dynamics. We have previously demonstrated the properties of the algorithm without the assumption of convexity in \cite{East2018}. Here, we show the comparative performance against DP in a convex formulation through simulation.

The paper is organised as follows: Chapter 2 describes an MPC framework for energy management, and Chapter 3 provides a proof of convexity for the associated optimisation problem. Chapter 4 details the DP and ADMM algorithms, Chapter 5 presents numerical simulations comparing their performance, and Chapter 6 provides conclusions. 


\section{Model Predictive Control Framework}
At each time-instant that the control variable is updated, we assume that a prediction is made of future vehicle velocity, $\hat{v} = (\hat{v}_0,\dots,\hat{v}_{N-1})$, and road gradient, $\hat{\theta} = (\hat{\theta}_0,\dots,\hat{\theta}_{N-1})$, at a sampling frequency of 1Hz and a horizon length of $N$. The first derivative of $\hat{v}$ is approximated using central difference, and the predicted driver power demand $\hat{P}_{drv} = (\hat{P}_{drv,0},\dots,\hat{P}_{drv,N-1})$ can be calculated using a backwards-causal longitudinal model
\begin{equation} \label{equation_driver_power_demand}
\begin{aligned}
\hat{P}_{drv,k} =& \bigl[ m\dot{\hat{v}}_k + \frac{1}{2}\rho_a \hat{v}_k^2c_dA + c_{r}mg \cos \hat{\theta}_k + mg \sin \hat{\theta}_k \bigr]\hat{v}_k,
\end{aligned}
\end{equation}
where $m$ is the vehicle mass, $\rho_a$ is air density, $c_d$ is the drag coefficient, $A$ is the frontal area, $c_r$ is the rolling resistance coefficient, and $g$ is the acceleration due to gravity. 
 
Figure \ref{fig_powertrain_model} shows a simplified diagram of a parallel plug-in hybrid electric vehicle powertrain. Fuel power, $P_f$, is fed to an internal combustion engine that delivers mechanical power, $P_{eng}$, through a clutch, which can be disengaged so that the engine can idle (or be turned off) and the vehicle can be driven in an all-electric mode. Chemical power, $P_b$, from the battery is delivered as electrical power, $P_c$, through a circuit to a motor that delivers mechanical power, $P_{em}$, to the drivetrain. This flow of power can be reversed to charge the battery. It is assumed that all drivetrain components are 100\% mechanically efficient, so the motor and engine output powers are combined additively through a mechanical coupling device to provide the driver power demand, $P_{drv}$. We ignore the effect of high speed dynamics such as clutch engagement and engine lag. 

\begin{figure}
\centering
\scalebox{0.72}{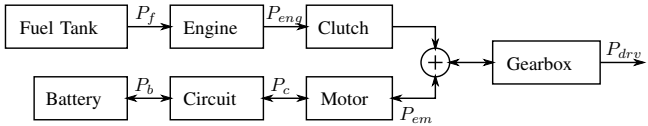}
\caption{Powertrain Model}
\label{fig_powertrain_model}
\end{figure} 

A controller is required to regulate the behaviour of this system as it is under-constrained in three degrees of freedom: the fraction of total driver demand power delivered from the engine, the clutch engagement, and the gear selection. Both the clutch disengagement and gear selection (assuming a discrete variable transmission) introduce integer decision variables that make the resulting control problem non-convex (and significantly more complex), so we assume here that these variables are determined by heuristics.

We assume a simple shifting strategy by which the powertrain speed $\omega_k$ can be found from
\begin{equation*} \label{equation_gear_selection}
\omega_k = k_i v_k, \hspace{5mm} \text{if} \hspace{5mm} \underline{v}_i \leq v_k \leq \overline{v}_i, \ i=1,\dots,N_g,
\end{equation*}
where $k_i$ are constants and $N_g$ is the number of gear ratios. We also assume that when the vehicle is braking, the clutch is open, the engine is idle, and the brakes provide a fixed fraction of the braking force $\gamma$. Finally, we assume that when the demand power is positive and the powertrain speed is lower than the minimum engine speed $\underline{\omega}_{eng}$, the clutch is open and the electric motor provides all of the driving power. Therefore, the engine power fraction only needs to be calculated when the driver power demand is positive, and the powertrain speed is greater than the minimum engine speed. These statements define the sets:
\begin{gather*}
\P = \{k | \hat{P}_{drv,k}>0, \hat{\omega}_k \geq \underline{\omega}_{eng} \}, \\
\C = \{ k | \hat{P}_{drv,k}>0 , \hat{\omega}_k < \underline{\omega}_{eng}  \}, \ 
\B = \{ k | \hat{P}_{drv,k}\leq 0 \},
\end{gather*}
from which predicted motor and engine speeds are given by
\begin{gather*}
\hat{\omega}_{em,k} = \hat{\omega}_k \  \forall \ k, \ \hat{\omega}_{eng,k} = \begin{cases} \hat{\omega}_k &  k \in \P \\
\underline{\omega}_{eng} &  k \notin \P \end{cases}.
\end{gather*}
The engine loss map is modelled as a set of convex quadratic functions of $P_{eng}$ of the form
\begin{equation*}
P_{f}(P_{eng},\omega_{eng}) = \alpha_{2,\omega_{eng}} P_{eng}^2 + \alpha_{1,\omega_{eng}} P_{eng} + \alpha_{0,\omega_{eng}},
\end{equation*}
for $\omega_{eng}$ in the interval $\underline{\omega}_{eng},\dots,\overline{\omega}_{eng}$, and the motor loss map is modelled using convex quadratic functions as
\begin{equation*}
P_{c}(P_{em},\omega_{em}) = \beta_{2,\omega_{em}} P_{em}^2 + \beta_{1,\omega_{em}} P_{em} + \beta_{0,\omega_{em}},
\end{equation*}
for $\omega_{em}$ in the interval $\underline{\omega}_{eng},\dots,\overline{\omega}_{eng}$. Loss functions are then found of the form 
\begin{alignat*}{3}
&\hat{P}_{f,k} &&= f_k(\hat{P}_{eng,k}) &&= \alpha_{2,k}\hat{P}_{eng,k}^2 + \alpha_{1,k}\hat{P}_{eng,k} + \alpha_{0,k}, \\
&\hat{P}_{c,k} &&= h_k(\hat{P}_{em,k}) &&= \beta_{2,k}\hat{P}_{em,k}^2 + \beta_{1,k}\hat{P}_{em,k} + \beta_{0,k},
\end{alignat*}
by linear interpolation between the loss map coefficients on the basis of $\hat{\omega}_{eng,k}$ and $\hat{\omega}_{em,k}$. The battery dynamics are 
\begin{equation*}
\hat{P}_{b,k} = g_k(\hat{P}_{em,k}) = \frac{V_{oc}^2}{2R} \left( 1 - \sqrt{1 - \frac{4R}{V_{oc}^2} h_k(\hat{P}_{em,k}) } \right)
\end{equation*}
where we assume constant $V_{oc}$ and $R$ for convexity. Limits on engine power are given for $k \in \P$ as
\begin{equation*}
\begin{aligned}
&\underline{\hat{P}}_{eng,k} &&= \max \{\underline{T}_{eng}\hat{\omega}_k, \hat{P}_{drv,k}- \min \{ \overline{T}_{em}\hat{\omega}_k, r^+_k \} \} \\
&\overline{\hat{P}}_{eng,k} &&= \min \{\overline{T}_{eng}\hat{\omega}_k, \hat{P}_{drv,k}-\underline{T}_{em}\hat{\omega}_k \},
\end{aligned}
\end{equation*}
where $\underline{T}$ and $\overline{T}$ are lower and upper limits on engine and motor torque, $r^+_k$ is the largest real root of $1 - \frac{4R}{V_{oc}^2} h_k(\hat{P}_{em,k})$, and upper and lower limits on battery state of charge are given by $\overline{E}$ and $\underline{E}$.
The MPC optimisation is then given by
\begin{equation}\label{eqn_opt}
\begin{aligned}
\hat{P}_{eng}^\star = \arg \min_{\hat{P}_{eng}} & \sum_{k=0}^{N-1} f_k(\hat{P}_{eng,k}) \\
\text{s.t.} \hspace{3mm} & \hat{E}_0 = E(t) \\
& \begin{aligned} &\hat{E}_{k+1} = \hat{E}_{k} -  g_k(\hat{P}_{em,k}) \\ 
&\underline{E} \leq \hat{E}_{k+1} \leq \overline{E}
\end{aligned} \Biggr\} && \forall \hspace{2mm} k \\
&  \begin{aligned}
&\hat{P}_{drv,k} = \hat{P}_{eng,k} + \hat{P}_{em,k} \\
&\underline{\hat{P}}_{eng,k} \leq \hat{P}_{eng,k} \leq \overline{\hat{P}}_{eng,k}
\end{aligned} \Biggr\} &&k \in \P \\
&\hat{P}_{eng,k} = \underline{\hat{P}}_{eng,k} &&k \notin \P \\
&\hat{P}_{em,k} = \hat{P}_{drv,k} &&k \in \C \\
&\hat{P}_{em,k} = \gamma \hat{P}_{drv,k} &&k \in \B\\
\end{aligned}
\end{equation}
and the first element of $\hat{P}^\star_{eng}$ is then applied as the control input to the engine.

\section{CONVEXITY}

We now demonstrate a convex reformulation of (\ref{eqn_opt}). Firstly, we restrict the domain of each $f_k(\hat{P}_{eng,k})$ and $g_k(\hat{P}_{em,k})$ so that they are non-decreasing, i.e impose new lower limits:
\begin{alignat*}{2}
&\underline{\hat{P}}^+_{eng,k} &&= \max \{ \underline{\hat{P}}_{eng,k} ,\ -\frac{\alpha_{1,k}}{2\alpha_{2,k}}   \}, \\
&\underline{\hat{P}}^+_{em,k} &&= \max \{ \hat{P}_{drv,k} - \overline{\hat{P}}_{eng,k} ,\ -\frac{\beta_{1,k}}{2\beta_{2,k}}   \}.
\end{alignat*}
This is expected as an increase in motor or engine output power would require an increase in input power, and it is an assumption also made in \cite{Egardt2014}. This domain restriction ensures that $g_k(\cdotp)$ is a one-to-one function, so
\begin{equation*}
\hat{P}_{b,k} = g_k(\hat{P}_{em,k}) \Leftrightarrow \hat{P}_{em,k} = g_k^{-1}(\hat{P}_{b,k}),
\end{equation*}
where $g_k^{-1}(\hat{P}_{b,k})$ is given by
\begin{align*}
g_k^{-1}(\hat{P}_{b,k}) = -\frac{\beta_{1,k}}{2\beta_{2,k}} + \sqrt{-\frac{R \hat{P}_{b,k}^2}{\beta_{2,k}V_{oc}^2} + \frac{\hat{P}_{b,k} - \beta_{0,k}}{\beta_{2,k}} + \frac{\beta_{1,k}^2}{4\beta_{2,k}^2}}.
\end{align*}
\begin{lemma}
\textit{Under the assumptions on $g_k(\cdotp)$ (convex, twice differentiable, non-decreasing, one-to-one), the inverse map $g_k^{-1}(\cdotp)$ is concave, twice differentiable, and increasing.}
\end{lemma}
\begin{proof}
The assumptions on $g_k(\cdotp)$ imply that $g'_k(x) \geq 0, \ g''_k(x) \geq 0 $ (where $g'$ and $g''$ are the first and second derivatives of $g$), and, if $y = g_k(x)$, then $x = g_k^{-1}(y)$ is unique for all $x$. Therefore
\begin{align*}
\frac{\text{d}y}{\text{d}x} = g'_k(x), \ & \frac{\text{d}x}{\text{d}y} = \frac{1}{g'_k(x)} = \frac{1}{g'_k(g_k^{-1}(y))} =(g^{-1}_k)'(y) > 0
\end{align*}
for all $y$, which demonstrates that $g_k^{-1}(\cdotp)$ is differentiable and monotonically increasing, and it is known that
\begin{align*}
\frac{\text{d}^2x}{\text{d}y^2} &= \frac{-1}{(g'_k(x))^2} \frac{\text{d}}{\text{d}y} \left( g'_k(x) \right) \\ &= \frac{-1}{(g'_k(x))^2}g''_k(x)\frac{\text{d}x}{\text{d}y} = \frac{-g''_k(g^{-1}_k(y)))}{\left( g'_k(g^{-1}_k(y)) \right)^3} \leq 0
\end{align*}
for all $y$, which demonstrates that $g_k^{-1}(\cdotp)$ is twice differentiable and concave.
\end{proof}
\bigskip

Using this definition of $g_k^{-1}(\cdotp)$, the cost function in problem (\ref{eqn_opt}) can be rewritten as
\begin{equation*}
f_k(\hat{P}_{eng,k}) = f_k( \hat{P}_{drv,k} - g_k^{-1}( \hat{P}_{b,k} )).
\end{equation*}

\begin{lemma} \label{lemma_1}
\textit{Under the assumptions on $f_k(\cdotp)$ (convex, non-decreasing), the function $f_k( \hat{P}_{drv,k} - g_k^{-1}( \hat{P}_{b,k} ))$ is convex and non-increasing with $\hat{P}_{b,k}$.}
\end{lemma}

\begin{proof}
The composition of a non-decreasing function $f_k(\cdotp)$ and decreasing function $ \hat{P}_{drv,k} - g_k^{-1}( \hat{P}_{b,k} )$ is necessarily non-increasing. The concave property of $g^{-1}_k(\cdotp)$ implies that for any $x_1, x_2 \in \mathcal{X}$ (where $\mathcal{X}$ is a domain such that $g^{-1}_k(x)$ is concave), and $\lambda \in [0,1]$,
\begin{align*}
&P_{drv,k} - g^{-1}_k\left( \lambda x_1 + (1-\lambda )x_2 \right) \\ 
&\quad\leq P_{drv,k} - \lambda g^{-1}_k(x_1) - (1-\lambda )g^{-1}_k(x_2).
\end{align*}
Therefore, using the properties that $f_k(\cdotp)$ is non-decreasing and convex, we have
\begin{align*}
& f_k\left( P_{drv,k} - g^{-1}_k\left( \lambda x_1 + (1-\lambda )x_2 \right) \right) \\
& \ \leq  f_k\left( P_{drv,k} - \lambda g^{-1}_k(x_1) - (1-\lambda )g^{-1}_k(x_2)  \right) \\
& \ \leq  \lambda f_k\left( P_{drv,k} - g^{-1}_k(x_1) \right) + (1-\lambda) f_k\left( P_{drv,k} - g^{-1}_k(x_2) \right) 
\end{align*}
for all $\lambda \in [0,1]$, which demonstrates that $f_k(\hat{P}_{drv,k} - g^{-1}(\hat{P}_{b,k}))$ is convex in $\hat{P}_{b,k}$.
\end{proof}
\bigskip
Problem (\ref{eqn_opt}) can now be re-written as
\begin{equation}\label{eqn_opt_reformulation}
\begin{aligned}
\hat{P}_{b}^\star = \arg \min_{\hat{P}_{b}} & \sum_{k\in\P} f_k( \hat{P}_{drv,k} - g^{-1}_k( \hat{P}_{b,k} ))  \\
\text{s.t.} \hspace{3mm} & \hspace{1mm} \hat{E}_0 = E(t) \\
&\begin{rcases} \hat{E}_{k+1} = \hat{E}_{k} -  \hat{P}_{b,k} \\ 
\underline{E} \leq \hat{E}_{k+1} \leq \overline{E} \\
\underline{\hat{P}}_{b,k} \leq \hat{P}_{b,k} \leq \overline{\hat{P}}_{b,k}
\end{rcases} \forall  k .
\end{aligned}
\end{equation}
This is now a convex optimisation problem, where 
\begin{equation*} \label{constraints_Pb}
\begin{aligned}
&\begin{drcases}
\underline{\hat{P}}_{b,k} = g_k( \underline{\hat{P}}_{em,k}^+ ) \\
\overline{\hat{P}}_{b,k} = g_k( \hat{P}_{drv,k} - \underline{\hat{P}}_{eng,k}^+ )
\end{drcases} k \in \P \\
&  \hspace{1mm} \underline{\hat{P}}_{b,k} = \overline{\hat{P}}_{b,k} = \begin{cases} g_k ( \hat{P}_{drv,k} ) & k \in \C \\ g_k ( \gamma \hat{P}_{drv,k} ) &k \in \B  \end{cases}.  
\end{aligned}
\end{equation*}
Existing analyses of convexity for similar optimal energy management problems replace the system dynamic equations, which appear in (\ref{eqn_opt}) as equality constraints, with inequalities, and then show that for a given solution the dynamic equations will be satisfied with equality \cite{Egardt2014}. The formulation presented here is superior for this application as it renders the dynamics as a linear system, which simplifies the implementation of convex optimisation algorithms.

\section{OPTIMISATION METHODS}

From Lemma \ref{lemma_1} it is known that $f_k( \hat{P}_{drv,k} - g_k^{-1}( \hat{P}_{b,k} ) )$ is non-increasing in $\hat{P}_{b,k}$, so it can be shown that if $\underline{E}_k \leq E_k - \overline{\hat{P}}_{b,k} \leq \overline{E}_k$ $\forall k$, then the minimising argument of (\ref{eqn_opt_reformulation}) is given trivially by $\hat{P}_b^\star = \overline{\hat{P}}_{b}$, and $\hat{P}_{eng,k}^\star = \hat{P}_{drv,k} - g_k^{-1}( \hat{P}_{b,k}^\star )$ for $k \in \P$. If this is not the case, an optimisation procedure must be applied.

Previous research on convex optimisation for the energy management problem has used general convex optimisation solvers to obtain the control inputs; here, we instead implement an ADMM algorithm similarly to \cite{East2018}. To demonstrate the relative performance characteristics of the algorithm we compare it in simulation with DP.

\subsection{Alternating Direction Method of Multipliers}

First, we introduce a dummy variable $\zeta$, and re-write (\ref{eqn_opt_reformulation}) as
\begin{equation}\label{eqn_opt_reformulation_2}
\begin{aligned}
\hat{P}_b^\star = \arg \min_{\hat{P}_{b}} & \sum_{k\in\P} f_k( \hat{P}_{drv,k} - g^{-1}_k( \hat{P}_{b,k} ))  + \Delta(\hat{P_b},\hat{E})  \\
\text{s.t.} \hspace{2mm} & \zeta = - \hat{P}_b \\
&\hat{E} = \Phi \hat{E}_0 + \Psi \zeta,
\end{aligned}
\end{equation}
where $\Phi$ is an $N-1$ column of ones, $\Psi$ is an $N-1 \times N-1$ lower triangular matrix of ones, and an indicator function is defined by
\begin{gather*}
\Delta(\hat{P_b},\hat{E}) = \sum_{k=0}^{N-1} \delta^{\hat{P}_b}_k(\hat{P}_{b,k})+ \sum_{k=1}^{N} \delta^{\hat{E}}_k(\hat{E}_{k}), \\
\delta^{x}_k(x_k) = \begin{cases} 0 & \underline{x}_k \leq x_k \leq \overline{x}_k \\ \infty &\text{otherwise}. \end{cases}
\end{gather*}
The augmented Lagrangian associated with problem (\ref{eqn_opt_reformulation_2}) is 
\begin{equation}\label{eqn_aug_lag}
\begin{aligned}
L(\hat{P}_b,\zeta,\hat{E},\lambda,\nu) = & \sum_{k\in\P} f_k( \hat{P}_{drv,k} - g^{-1}_k( \hat{P}_{b,k} ))+  \Delta(\hat{P_b},\hat{E})\\
 &  + \frac{\rho_1}{2} \|\hat{P}_b + \zeta + \nu \|^2  \\
& + \frac{\rho_2}{2}\|\Phi \hat{E}_0 + \Psi \zeta - \hat{E} + \lambda \|^2,
\end{aligned}
\end{equation}
where $\rho_1,\rho_2 >0$ are constants discussed in Section~\ref{sec:opt_params}.
We define projection functions for $x = [\, x_1 \ \cdots \ x_N]^\top$ as 
\begin{align*}
\pi^{x}_k(x_k) &= \min \{ \overline{x}_k, \max \{ \underline{x}_k , x_k \} \} , \\
\Pi^x(x) &= \begin{bmatrix} \pi^{x}_1(x_1) & \cdots & \pi^{x}_N(x_{N})\end{bmatrix}^\top,
\end{align*}
then the ADMM iteration is given by
\begin{subequations}\label{eqn_ADMM_iteration}
\begin{alignat*}{4}
&\hat{P}_{b,k}^{j+1} &&=&&  \pi_k^{\hat{P}_b} [ \arg \min_{\hat{P}_{b,k}} f_k( \hat{P}_{drv,k} - g^{-1}_k( \hat{P}_{b,k} )) \label{subeqn1}\\
& && &&+ \frac{\rho_1}{2}(\hat{P}_{b,k} + \zeta_k^j + \nu_k^j )^2 ] && \ \ k \in \P \\
&\hat{P}_{b,k}^{j+1} &&=&& \underline{\hat{P}}_{b,k} && \ \ k \notin \P \\
&\zeta^{j+1} &&=&& \left( \rho_1 I + \rho_2 \Psi^\top \Psi \right)^{-1}
[ -\rho_1(\hat{P}_b^{j+1} + \nu^j )  \\
& && && - \rho_2 \Psi^\top ( \Phi \hat{E}_0 - \hat{E}^j + \lambda^j ) ]  \\
&\hat{E}^{j+1} &&=&& \Pi^{\hat{E}} [ \Phi \hat{E}_0 + \Psi \zeta^{j+1} + \lambda^j ]  && \\
&\lambda^{j+1} &&=&& \lambda^j + \Phi \hat{E}_0 + \Psi \zeta^{j+1} - \hat{E}^{j+1}  \\
&\nu^{j+1} &&=&&  \nu^j + \hat{P}_b^{j+1} + \zeta^{j+1},
\end{alignat*}
\end{subequations}
which is initialised with the values
\begin{align*}
&\hat{P}_{b}^0 = \overline{\hat{P}}_{b}, \ \zeta^0 = -\hat{P}_b^0, \ \hat{E}^0 = \Pi^E \left( \Phi \hat{E}_0 + \Psi \zeta^0 \right) \\
&\lambda^0 = \Phi \hat{E}_0 + \Psi \zeta^0 - \hat{E}^0, \ \nu^0 = 0.
\end{align*}

The $\zeta$ variable is included in the formulation to ensure that the $\hat{P}_b$ update is separable, and the individual $\hat{P}_{b,k}$ updates are convex optimisation problems that we solve using an unconstrained Newton method with a backtracking line-search. The matrix inversion associated with the $\zeta$ update can be computed off-line as it involves no decision variables, and the remaining variable updates are trivial. This iteration is repeated until the residuals $\|r^{j+1}\|_2 \leq \epsilon$ and $\| s^{j+1} \|_2 \leq \epsilon$, where $\epsilon$ is the convergence threshold, and 
\begin{align*}
r^{j+1} &= \left[ \begin{array}{c} \hat{P}_{b}^{j+1} \\ \Phi \hat{E}_0  \end{array} \right] + \left[ \begin{array}{c c} I & 0 \\ \Psi & -I   \end{array} \right] \left[ \begin{array}{c} \zeta^{j+1} \\ \hat{E}^{j+1} \end{array} \right]  \\
s^{j+1} &= \left[ \begin{array}{c c} \rho_1 I & 0 \\ \rho_2 \Psi & -\rho_2I \end{array} \right] \left[ \begin{array}{c} \zeta^{j} - \zeta^{j+1} \\ \hat{E}^{j} - \hat{E}^{j+1} \end{array} \right].
\end{align*}
In \cite{East2018}, we demonstrate that this algorithm will converge to a point satisfying the first order conditions of problem (\ref{eqn_opt_reformulation}), which must be the optimal point in this convex formulation. The solution for problem (\ref{eqn_opt}) is then given by
\begin{align*}
\hat{P}^\star_{eng,k} &= \hat{P}_{drv,k} - g^{-1}_k( \hat{P}_{b,k}  ) \ k \in \P.
\end{align*}

\subsection{Dynamic Programming}
Although it is well known that DP is computationally expensive and unlikely to perform as well as a dedicated convex optimisation algorithm, we use it here to demonstrate baseline performance as it is commonly used to solve various forms of the energy management problem and is guaranteed to provide the optimal solution for a sufficiently high mesh density. 
In the absence of other tailored algorithms for this problem, a comparison with a general purpose convex optimization solver would be similarly flawed, as the results would be obfuscated by the performance of the differing software implementations. 
%

\begin{definition}
Define $J \in \mathbb{R}^{N_E \times N+1}$ as a matrix of costs, where $J[E,k]$ is the cost-to-go at state index $E=\underline{E},\dots,\overline{E}$ (we assume the values of $E$ are evenly spaced) and timestep index $k = 0,\dots,N$. The function $J( E,k)$ returns the cost-to-go, linearly interpolated from the available values of $E$ and $k$, and returns $\infty$ outside the feasible range of $E$. We use the notation $J[:,k]$ to refer to column $k$ of $J$.
\end{definition}
\begin{definition}
Define $U \in \mathbb{R}^{N_E \times N+1}$ as a matrix of optimal control inputs, where the value $U[E,k]$ is the optimal control input at state index $E$ and timestep index $k$. The function $U(E,k) $ returns the optimal control input linearly interpolated from the available values of $E$ and $k$, and the notation $U[:,k]$ refers to the column $k$ of $U$.
\end{definition}
\begin{definition}
Define $P_d \in \mathbb{R}^{N_P}$ as the vector with elements equal to the discrete, evenly-spaced values that may be assigned to the control input $P_{eng,k}$, such that $P_d = (\underline{P}_{eng,k}^+,\dots,\overline{P}_{eng,k})= (P_{d,1},\dots,P_{d,N_P})$ and $P_d[i]=P_{d,i}$.
\end{definition}

The algorithm is initialised by setting $U[:,N]=J[:,N]=0$, then for $E = \underline{E},\dots,\overline{E}$ the cost to go is calculated at $k=N-1,N-2,\dots,0$ by solving
\begin{align*}
J[E,k] = \begin{cases} \begin{aligned} &\min_i \big[f_k(P_d[i]) \\
&+ J ( E-g_k(P_d[i]) ,k+1) \big] \end{aligned} &k \in \P \\
f_k(P_d[1]) + J ( E-\underline{\hat{P}}_{b,k} ,k+1) &k \notin \P. 
\end{cases}
\end{align*}
The integer $i$ that minimises this function is found by evaluating all possible values, and selecting the minimum. The same method is used to evaluate
\begin{align*}
U[E,k] = \begin{cases} \begin{aligned} &\underset{i}{\arg \min} [f_k(P_d[i]) \\
&+ J (E-g_k(P_d[i]) ,k+1) ]\end{aligned}  &k \in \P \\
P_d[1] &k \notin \P. \end{cases}
\end{align*}
The optimal trajectory and control inputs can be found from 
\begin{equation*}
\begin{aligned}
\hat{E}_{k+1} &= \begin{cases} \hat{E}_{k} - g_k(U ( \hat{E}_{k},k )) & k \in \P \\
\hat{E}_k - \underline{\hat{P}}_{b,k} &k \notin \P \end{cases}, \\
\hat{P}_{eng,k}^* &= \begin{cases} U ( \hat{E}_{k},k ) & k \in \P \\
\underline{\hat{P}}_{eng,k} & k \notin \P \end{cases}.
\end{aligned}
\end{equation*}
for $k=0,\dots,N-1$.

\section{SIMULATIONS}\label{section_simulations}
We are interested in the comparative performance of ADMM and DP and not the properties of the MPC framework itself, so we only consider single-shot instances of the MPC optimisation with assumed driver demand profiles. 

We use the FTP-75 cycle to define a vector, $v$, of velocity samples, and generate a vector $\theta$, of gradient samples
\begin{equation*}
v = (v_0,\dots,v_T), \ \theta = (\theta_0,\dots,\theta_T), \ \theta_t = \begin{cases} \theta_s & t \leq T/2 \\ -\theta_s & t > T/2  \end{cases},
\end{equation*} 
for $t=0,\dots,T$ at 1Hz, where $\theta_s$ is a constant defined for each optimisation. The predicted velocity and gradient vectors used for each problem instance are given by
\begin{equation*}
\hat{v} = (v_{(1-\mu) T},\dots,v_T), \ \hat{\theta} = (\theta_{(1-\mu) T},\dots,\theta_T),
\end{equation*}
where $(1-\mu) T$ is rounded to the nearest integer, and $\mu \in [0,1]$. The predicted power demand is then calculated from (\ref{equation_driver_power_demand}). Essentially, $\theta_s$ is a variable used to modulate the magnitude of the predicted power, and $\mu$ is a variable used to modulate the length of the predicted power vector.

  We use a vehicle model to represent a 1800kg parallel plug-in hybrid passenger vehicle with a 100kW petrol internal combustion engine, a 50kW electric motor, and a 21.5Ah lithium-ion battery. We assume a fixed open circuit voltage of 350V and internal resistance of $0.1\Omega$ for the battery, and the fixed braking fraction, $\gamma$, is set at 0.4. An initial state of charge $\hat{E}_0= (50 + 10\mu)\%$ is used to represent battery discharge as the drive cycle is shortened, and upper and lower limits on state of charge are set as 62\% and 50\%.  

\subsection{Optimisation Parameters}\label{sec:opt_params}

  For DP we assume that $N_P = N_E/10$, as the feasible battery storage band is approximately an order of magnitude greater than the maximum power available from the motor, and this ensures that the $P_d$ and $J$ meshes have similar numerical density. We then use the value of $N_E$ to modulate the desired accuracy of the solution. From numerical experiments we find that $N_E=1000$ is sufficient to achieve a high accuracy approximation of the global minimum, and all further references to the global minimum/optimum refer to control inputs and state trajectories that are found this way.

For the ADMM algorithm, $\rho_1$ and $\rho_2$ must be tuned to match the dynamics of the optimisation problem, and we tune these parameters to the instance where $\theta_s=2^\circ$ and $\mu=1$. First we obtain the global minimum, then run 300 ADMM iterations and record a cost that measures the deviation of the fuel consumption vector at algorithm termination, $\hat{P}_f^\dagger$, from the global optimum fuel consumption vector, $\hat{P}_f^\star$, with
\begin{equation}\label{simulation_cost}
J = \sum_{k=0}^{N-1} | \hat{P}_{f,k}^\star - \hat{P}_{f,k}^\dagger |.
\end{equation}
The results shown in Figure \ref{figure3} illustrate that for this particular instance of the optimisation problem, variations of $\rho_1$ within two orders of magnitude and variations of $\rho_1$ within one order of magnitude induce little suboptimality in the resulting optimisation. We use the parameters $\rho_1=2.34 \times 10^{-4}$ and $\rho_2=8.86 \times 10^{-9}$ for subsequent ADMM optimisations. The value $\rho_2 \approx 10^{-8}$ may appear low, however it is appropriate given the relative scaling of the fuel consumption for each timestep ($\approx 10$kJ) and the battery capacity ($\approx 10$MJ) in (\ref{eqn_aug_lag}).

\begin{figure}
\begin{center}
\includegraphics[scale=0.85]{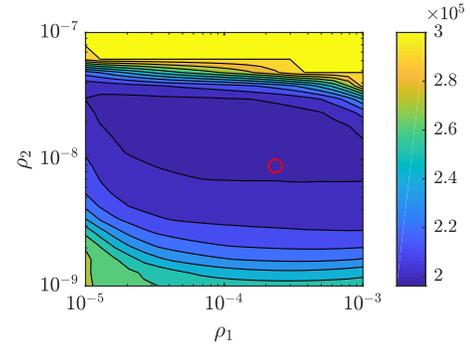}
\caption{$J$ evaluated for $10^{-5} \leq \rho_1 \leq 10^{-3}$ and $10^{-9} \leq \rho_2 \leq 10^{-7}$. The color scale is saturated at $J=3\times 10^{5}$, and the minimum value is highlighted by the red circle.  \label{figure3}}
\end{center}
\vspace{1mm}
\end{figure}

\subsection{Optimisation with variation in power demand}
A practically useful algorithm should demonstrate consistent performance across a broad range of optimisation scenarios, and to test this we generate drive cycles with varied magnitudes using $\mu=1$ and $-2^\circ \leq \theta_s \leq 2^\circ$, then solve using DP with $30 \leq N_E \leq 120$ and ADMM with $10^5 \leq \epsilon \leq 10^6$. These ranges of $N_E$ and $\epsilon$ are chosen to produce solutions with comparable accuracy that are computationally tractable. We run each instance of both algorithms 10 times and record the average time for completion and the deviation from the optimum cost measured by (\ref{simulation_cost}). The results are shown in Fig.~\ref{figure4}. 

\begin{figure}
\begin{center}
\includegraphics[scale=0.8]{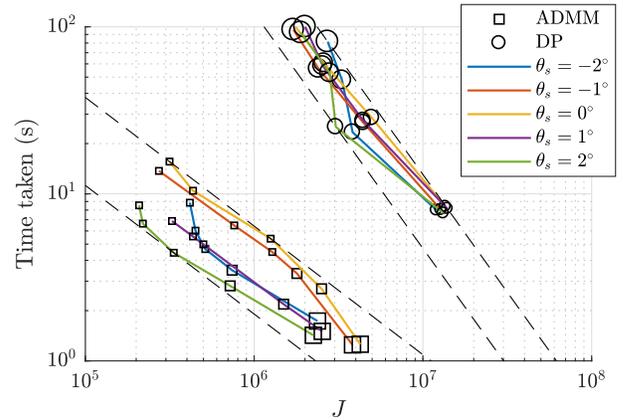}
\caption{ Mean time taken against optimality measured by $J$ with $\mu=1$. The ADMM marker size corresponds to the magnitude of $\epsilon$, and DP marker size corresponds to the magnitude of $N_E$. The dashed lines show the boundaries of the minimum width linear bands that contain the results. \label{figure4}}
\end{center}
\vspace{-4mm}
\end{figure}

The ADMM algorithm solves the optimisation fastest with $\theta_s=2^\circ$ (the cycle it is tuned to) and is slowest to converge for $\theta_s=0^\circ$. Between these scenarios, all results lie within a linear band with a maximum variation in computation time of less than an order of magnitude, displaying a degree of robustness to variations in drive cycle. Comparatively, the results for DP display more consistency as they are contained within a tighter band, but also show worse performance: for the range of parameters tested, DP takes approximately two orders of magnitude more time for a comparative level of accuracy. For example, using the $\theta_s=1^\circ$ cycle, ADMM terminates with $J=2.5\times 10^6$ after $1.5s$ with $\epsilon=10^6$, whereas DP takes $115.8s$ to achieve $J=2.0\times 10^6$ using $N_E=120$. The bands also diverge, suggesting that ADMM performs even better for higher accuracy, although it is not clear to what extent this relationship can be extrapolated.

\subsection{Optimisation with variation in prediction horizon}
To investigate robustness to variations in horizon length, we generate drive cycles with $\theta_s=-2^\circ$ and $\mu \in [0.25,1]$, then solve the optimisation and record parameters using the same method as above. The results are shown in Figure \ref{figure5}. 

\begin{figure}
\begin{center}
\includegraphics[scale=0.8]{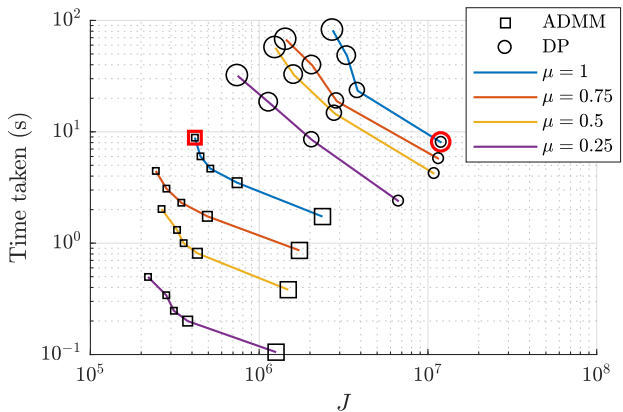}
\caption{Mean time taken against optimality measured by (\ref{simulation_cost}) with $\theta_s=2^\circ$. The size of the ADMM marker corresponds to the magnitude of $\epsilon$, and the size of the DP marker corresponds to the magnitude of $N_E$. The state trajectories for the points highlighted in red are shown in Figure \ref{figure6}.  \label{figure5}}
\end{center}
\end{figure}

It is shown that the time taken for the ADMM algorithm to terminate decreases dramatically as the length of the horizon is reduced. For example, the time taken with $\epsilon = 10^6$ is $1.8s$ for $\mu = 1$ and $0.9s$ for $\mu = 0.75$, then $0.1s$ for $\mu = 0.25$. As the time taken for DP decreases in a similar manner, the key observation is that the significant speed improvement with ADMM is sustained as the horizon is reduced. This suggests that if the algorithm can deliver a performance benefit for a broad range of drive cycle characteristics at the start of a journey, it will then also deliver the same performance benefit as the journey is completed.

In Figure \ref{figure6} we show state trajectories for the $\theta_s=-2^\circ$ and $\mu=1$ cycle that are obtained with DP and ADMM using a comparative level of computational time, as well as the global optimum. These plots clearly illustrate the superior performance of ADMM for this problem, but also show a significant limitation. Whilst the ADMM state trajectory follows the optimum much more closely, it is also shown to violate the SOC constraints by $\approx 0.02\%$. This is because there is no explicit guarantee in the ADMM termination criteria that the constraints are met: this is only ensured when the primal residual is exactly zero. We do not, however, believe that this is a significant problem with ADMM for this application, as for all optimisations shown in Figures \ref{figure4} and \ref{figure5}, no state of charge violation was greater that $1\%$, and this level of violation is likely to be lower in magnitude than the uncertainty in the current state of charge.

\begin{figure}
\begin{center}
\includegraphics[scale=0.8]{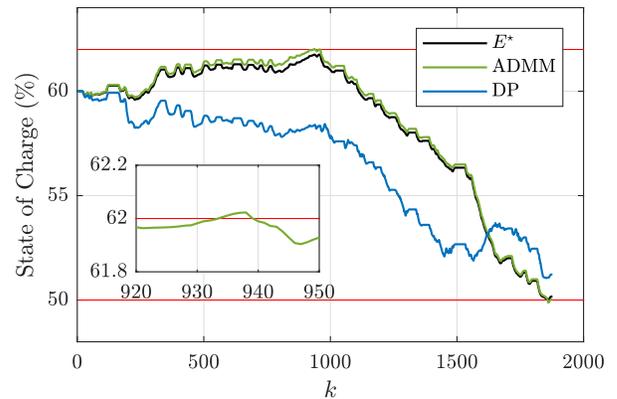}
\caption{Optimum state of charge trajectory plotted with ADMM and DP optimisations highlighted in red in Figure \ref{figure5}, where $\mu=1$, $\theta_s=-2^\circ$, $N_E = 120 $, and $\epsilon = 10^5$.  \label{figure6}}
\end{center}
\vspace{1.5mm}
\end{figure}


\section{CONCLUSION}

A convex formulation of the MPC optimisation associated with energy management in hybrid electric vehicles is presented, with a demonstration of convexity that allows the state dynamics to be modelled as a linear system. We propose an ADMM algorithm for the solution of this problem, and demonstrate approximately two orders of magnitude improvement in computation time compared to DP. A degree of robustness to variations in the magnitude and length of the predicted drive cycle is also demonstrated.

\bibliographystyle{ieeetr}
\bibliography{library}

\end{document}